\documentclass[11pt]{article}
\usepackage{amsmath, amssymb, amsthm,verbatim}

\date{}
\title{\vspace{-1cm} Constructing dense graphs with sublinear Hadwiger number}
\author{
Jacob Fox \thanks{Department of Mathematics,
Massachusetts Institute of Technology, Cambridge, MA 02139-4307. E-mail:
fox@math.mit.edu. Research supported by a Simons Fellowship and NSF grant DMS-1069197.}}

\newtheorem{theorem}{Theorem}
\newtheorem{lemma}{Lemma}

\newtheorem{definition}{Definition}
\newtheorem{proposition}[theorem]{Proposition}

\newtheorem{conjecture}{Conjecture}

\begin{document}
\maketitle

\begin{abstract}
Mader asked to explicitly construct dense graphs for which the size of the largest clique minor is sublinear in the number of vertices. Such graphs exist as a random graph almost surely has this property. This question and variants were popularized by Thomason over several articles. We answer these questions by showing how to explicitly construct such graphs using blow-ups of small graphs with this property.
This leads to the study of a fractional variant of the clique minor number, which may be of independent interest.
\end{abstract}

\section{Introduction}

A graph $H$ is a {\it minor} of a graph $G$ if $H$ can be obtained from a subgraph of $G$ by contracting edges.
 Minors form an important connection between graph theory, geometry, and topology. For example,
the Kuratowski-Wagner theorem states that a graph can be embedded in the plane if and only if it has neither the complete graph $K_5$
nor the complete bipartite graph $K_{3,3}$ as a minor. This example played an important role in the development of topological graph theory,
whose masterpiece is the Robertson-Seymour graph minor theorem. In a series of twenty papers \cite{RS1},
they proved Wagner's conjecture that every family of graphs closed under taking minors is characterized by a finite list of forbidden minors.

The {\it Hadwiger number} $h(G)$ of a graph $G$ is the order of the largest clique which is a minor of $G$. The famous conjecture of Hadwiger \cite{Ha} states that every graph of chromatic number $k$ has Hadwiger number at least $k$. Hadwiger proved his conjecture for $k \leq 4$. Wagner \cite{Wa} proved that the case $k=5$ is equivalent to the Four Color Theorem.
In a tour de force, Robertson, Seymour, and Thomas \cite{RST} settled the case $k=6$ also using the Four Color Theorem. The conjecture is still open for $k \geq 7$.

Bollob\'as, Catlin, and Erd\H{o}s \cite{BCE} analyzed the Hadwiger number of random graphs.
They showed that a random graph $G$ on $n$ vertices almost surely satisfies $h(G)$ is asymptotic to $\frac{n}{\sqrt{\log n}}$. Here, and throughout the paper, all
logarithms unless otherwise indicated are in base $2$. Also using the well known fact that the chromatic number of a random graph on $n$ vertices is almost surely $\Theta(n/\log n)$, they deduced that almost all graphs satisfy Hadwiger's conjecture.

Mader showed that large average degree is enough to imply a large clique minor. Precisely, for each integer $t$ there is a
constant $c(t)$ such that every graph $G$ of average degree at least $c(t)$ satisfies $h(G) \geq t$. Kostochka \cite{K,K84} and
Thomason \cite{Th} independently proved that $c(t)=\Theta(t\sqrt{\log t})$. Thomason \cite{Th01} later determined the asymptotic behavior of
 $c(t)$, with random graphs of a particular density as extremal graphs for this problem. Myers \cite{My} proved that any extremal graph under certain conditions
for this problem must be quasirandom.

Random graphs have some remarkable properties for which it is difficult to explicitly construct graphs with these properties. One well-known example
is Erd\H{o}s' lower bound on Ramsey numbers, which shows that almost all graphs on $n$ vertices do not contain a clique or independent set of order $2\log n$.
Despite considerable attention over the last 60 years, there is no known construction (in polynomial time) of a graph $G$ on $n$ vertices for which the largest
clique or independent set in $G$ is of order $O(\log n)$.

Another interesting property of random graphs already mentioned is that almost surely they do not contain a clique minor of linear size.
Mader asked to construct a dense graph on $n$ vertices with $h(G)=o(n)$. One of the main motivations for this problem is that proving interesting upper bounds on
the Hadwiger number of a graph appears to be a difficult problem. Thomason \cite{Th00} showed that many of the standard constructions of quasirandom graphs have
linear clique minors and therefore cannot be used to answer Mader's problem.

Mader's problem and a few variants were discussed by Thomason in several papers \cite{Th00,Th00a,Th01,Th04,Th06}
and also by Myers \cite{My}. Thomason \cite{Th00} posed the stronger problem of constructing a graph $G$ on $n$ vertices for which the Hadwiger numbers of $G$ and its complement $\bar G$ are both $o(n)$.
He speculates \cite{Th00a} that this problem might be as hard as the classical Ramsey problem
of finding explicit graphs $G$ such that both $G$ and $\bar G$ contain only small complete subgraphs.

Here we solve both the problems of Mader and Thomason. To do so, it is helpful to define what an explicit construction is. We view a graph on $2^n$ vertices as a function $f:{\{0,1\}^n \choose 2} \rightarrow \{0,1\}$, where the value of $f$ tells whether or not two vertices are adjacent. By an explicit construction we mean that the function $f$ is computable in polynomial time (in $n$). That is, given two vertices, we can compute whether or not they are adjacent in time polynomial in the number of bits used to represent the vertices. There is also a weaker notion of explicit graph which is sometimes used. In this version, the edges of the graph can be computed in time polynomial in the number of vertices of the graph.

Our construction which answers the problems of Mader and Thomason is given in {\it nearly constant} time using blow-ups of nearly constant size graphs. We show that if a (small) graph is dense and has
relatively small Hadwiger number, then its blow-ups also have this property.

We formally define the blow-up of a graph as follows. For graphs $G$ and $H$, the {\it lexicographic product} $G \cdot H$ is the graph
on vertex set $V(G) \times V(H)$, where $(u_1,v_1),(u_2,v_2) \in V(G) \times V(H)$ are adjacent if and only if $u_1$ is adjacent to $u_2$ in $G$,
or $u_1=u_2$ and $v_1$ is adjacent to $v_2$ in $H$.  Define The {\it blow-up} $G(t)=G \cdot I_t$, where $I_t$ is the empty graph on $t$ vertices.
Define also the {\it complete blow-up} $G[t]=G \cdot K_t$.

A graph $G$ we call {\it $\epsilon$-Hadwiger} if $h(G) \leq \epsilon |G|$. We will show for each $\epsilon>0$ how to construct a graph $G$
on $n(\epsilon)=2^{(1+o(1))\epsilon^{-2}}$ vertices in time $2^{(1+o(1))n(\epsilon)^2/2}$ such that every complete blow-up of $G$ and its complement are $\epsilon$-Hadwiger. Such blow-ups answer the questions of Mader and Thomason,
as one can compute any adjacency between vertices by simply looking at which parts of the blow-up the vertices belong. As the size of $G$ depends only on $\epsilon$,
the time to compute whether two vertices are adjacent is nearly constant for $\epsilon$ slowly tending to $0$.

The bounds above give an explicit construction of a graph on $N$ vertices for which the Hadwiger number of the graph
and its complement is at most $O(\frac{N}{\sqrt{\log \log \log N}})$. For the weaker notion of explicit construction, in which the running time is polynomial
in the number of vertices, the graph on $N$ vertices we obtain has the property that it and its complement has Hadwiger number at most
$O(\frac{N}{\sqrt{\log \log N}})$. While these bounds are sublinear, they do not come close to the tight bound of $O(\frac{N}{\sqrt{\log N}})$ which
almost all graphs on $N$ vertices satisfy.

In order to study the Hadwiger number of a blow-up of a graph, it will be helpful to define a fractional version of the Hadwiger number. This notion
had independently been introduced earlier by Seymour \cite{Se}. A {\it bramble}
$\mathcal{B}$ for a graph $G$ is a collection of connected subgraphs of $G$ satisfying each pair $B,B' \in \mathcal{B}$ share a vertex or there is an edge of $G$
connecting $B$ to $B'$.

\begin{definition}\label{thm2}
The {\it fractional Hadwiger number} $h_f(G)$ of a graph $G$ is the maximum $h$ for which there is a bramble $\mathcal{B}$ for $G$,
and a weight function $w:\mathcal{B} \rightarrow \mathbb{R}_{\geq 0}$ such that $h=\sum_{B \in \mathcal{B}}w(B)$
and for each vertex $v$, the sum of the weights of the subgraphs in $\mathcal{B}$ containing $v$ is at most $1$.
\end{definition}

Define a {\it strong bramble} for a graph $G$ to be a collection of connected subgraphs of $G$ satisfying for each pair $B,B' \in \mathcal{B}$ (with possibly $B=B'$)
there is an edge of $G$ connecting $B$ to $B'$.  We define the {\it lower fractional Hadwiger number} $h'_f(G)$ similarly,
except that $\mathcal{B}$ is required to be a strong bramble and not a bramble. The fractional Hadwiger number and the lower fractional Hadwiger number are closely related.
Indeed, it is easy to show that if $G$ has an edge, then $h_f(G)/2 \leq h'_f(G) \leq h_f(G)$. Equality occurs in the lower bound if $G$ is complete and the upper bound if $G$ is complete bipartite.

For a positive integer $r$, the {\it $r$-integral Hadwiger number} $h_r(G)$ is defined the same as the fractional Hadwiger number, but all weights have to be multiples of $1/r$. We similarly define the lower version $h'_r(G)$. Note that $h_1(G)=h(G)$, and if $s$ is a multiple of $r$, then $h_s(G) \geq h_r(G)$. It is easy to check that $h_f(G)=\lim_{r\to \infty} h_r(G)$, and it follows that
$h_f(G) \geq h(G)$.

The relationship between the Hadwiger number of the blow-up of a graph and the fractional Hadwiger number of the graph is demonstrated by the following simple proposition.

\begin{proposition}\label{simpleprop} For every graph $G$ and positive integer $r$, we have $$h(G[r])=r \cdot h_r(G) \leq r \cdot h_f(G).$$
\end{proposition}
Essentially the same proof also gives $h(G(r))=r \cdot h'_r(G) \leq r \cdot h'_f(G)$.

Thus, if we found a dense graph $G$ with relatively small fractional Hadwiger number, then the blow-up $G[r]$ would also be dense and
have relatively small Hadwiger number. To solve Mader's problem,
it therefore suffices to show that there are dense graphs $G$ on $n$ vertices with fractional Hadwiger number $h_f(G)=o(n)$.

It is not difficult to show that if $h(G)<4$, then $h_f(G)=h(G)$. However, there are planar graphs on $n$ vertices with $h(G)=4$ and $h_f(G)=\Theta(\sqrt{n})$. Indeed, consider the $\sqrt{n} \times \sqrt{n}$ grid graph. The grid graph is planar and thus has Hadwiger number at most $4$.
For each $i$, let $P_i$ denote the induced path consisting of the vertex $(i,i)$ and all vertices of the grid graph directly below or to the right of this point.
Assigning each of these $\sqrt{n}$ paths weight $1/2$, we get that the fractional Hadwiger number (and even the $2$-integral Hadwiger number) of this grid graph is
at least $\sqrt{n}/2$.  This example also shows that the following upper bound on the fractional Hadwiger number cannot be improved apart from the constant factor.

\begin{theorem}\label{amainthm}
If $G$ is a graph on $n$ vertices, then $$h_f(G) \leq \sqrt{2h(G)n}.$$
\end{theorem}

It is natural to study the fractional Hadwiger number of random graphs. Theorem \ref{amainthm}
implies that a random graph on $n$ vertices almost surely has fractional Hadwiger number $O(n/(\log n)^{1/4})$.
We prove a much better estimate, that the fractional Hadwiger number of a random graph is almost surely asymptotic to its Hadwiger number.
Bollob\'as, Catlin, and Erd\H{o}s \cite{BCE} showed that the random graph $G(n,p)$ on $n$ vertices with fixed edge probability $p$
almost surely has Hadwiger number asymptotic to $\frac{n}{\sqrt{\log_b n}}$, where $b=1/(1-p)$.

\begin{theorem}\label{fractrandom}
The fractional Hadwiger number of a random graph is almost surely asymptotically equal to its Hadwiger number. That is, for fixed $p$ and all $n$,
almost surely $$h_f(G(n,p))=(1+o(1))\frac{n}{\sqrt{\log_b n}},$$ where $b=1/(1-p)$.
\end{theorem}

We conjecture that a stronger result holds, that they are in fact almost surely equal.

\begin{conjecture}
A graph $G$ on $n$ vertices picked uniformly at random almost surely satisfies $h(G)=h_f(G)$.
\end{conjecture}

This would imply that for most graphs $G$, the ratio of the Hadwiger number of $G$ to the number of vertices of $G$ is the same as for its blow-ups.

\vspace{0.1cm}
\noindent {\bf Organization:} In the next section, we establish several upper bounds on the fractional Hadwiger number, including
Theorems \ref{amainthm} and \ref{fractrandom}. In Section \ref{sect3}, we use these upper bounds on the fractional Hadwiger number to construct dense graphs with
sublinear Hadwiger number. We finish with some concluding remarks. We sometimes omit floor and ceiling signs for clarity of presentation. 

\section{Upper bounds on the fractional Hadwiger number}\label{sect2}

In this section we establish several upper bounds on the fractional Hadwiger number of a graph. We begin by proving Theorem \ref{amainthm}, which states that if $G$ is a graph on $n$ vertices, then $h_f(G) \leq \sqrt{2h(G)n}$.

\vspace{0.1cm}
\noindent {\bf Proof of Theorem \ref{amainthm}:} Let $\mathcal{B}$ be a bramble for a graph $G$ on $n$ vertices.
Let $w:\mathcal{B} \rightarrow \mathbb{R}_{\geq 0}$ be a weight function such that $h=\sum_{B \in \mathcal{B}}w(B)$
and for each vertex $v$, the sum of the weights of the subgraphs in $\mathcal{B}$ containing $v$ is at most $1$.
It suffices to show that $\mathcal{B}$ contains a subcollection of at least $\frac{h^2}{2n}$ vertex-disjoint subgraphs.
Indeed, contracting these subgraphs we get a clique minor in $G$ of order at least $\frac{h^2}{2n}$,
and picking $\mathcal{B}$ and $w$ to maximize $h$, we have $h=h_f(G)$ so that $h(G) \geq \frac{h_f(G)^2}{2n}$ or equivalently $h_f(G) \leq \sqrt{2h(G)n}$.
We will prove the desired lower bound on the maximum number of vertex-disjoint trees in $\mathcal{B}$ by induction on $n$.
The base case $n=1$ clearly holds, and suppose the desired bound holds for all $n'<n$.

Let $B_0$ be a subgraph in $\mathcal{B}$ with the minimum number of vertices, and let $t=|B_0|$.
Since for each vertex $v$, the sum of the weights of the subgraphs in $\mathcal{B}$ containing $v$ is at most $1$,
summing this inequality over all vertices yields $$\sum_{B \in \mathcal{B}}w(B)|B| \leq n.$$ In particular, $$h =\sum_{B \in \mathcal{B}}w(B) \leq n/|B_0|=n/t.$$

Delete from $\mathcal{B}$ all subgraphs containing a vertex in $B_0$, and let $\mathcal{B}'$ be the resulting subcollection of subgraphs.
Since for each vertex $v$, the sum of the weights of the trees containing $v$ is at most $1$, we have $\sum_{B \in \mathcal{B}'}w(B) \geq h-t$.
The number of vertices not in $B_0$ is $n-t$. Hence, from a maximum subcollection of vertex-disjoint subgraphs in $\mathcal{B}'$ and adding $B_0$,
we get by induction at least $$1+\frac{(h-t)^2}{2(n-t)} \geq 1+\frac{(h-t)^2}{2n}=1+\frac{h^2}{2n}(1-\frac{t}{h})^2 \geq 1+\frac{h^2}{2n}(1-\frac{n}{h^2})^2 \geq 1+\frac{h^2}{2n}(1-\frac{2n}{h^2})=\frac{h^2}{2n}$$
vertex disjoint subgraphs in $\mathcal{B}$, which completes the proof.\qed

We next establish a useful lemma for proving Theorem \ref{fractrandom}. This lemma extends the result of Bollob\'as, Catlin, and Erd\H{o}s \cite{BCE} on the largest 
clique minor in a random graph by giving a bound on the size of the largest clique minor in which the size of the connected subgraphs corresponding to the vertices of the clique are bounded. Recall that a clique minor
in a graph $G$ of size $t$ consists of $t$ vertex disjoint connected subsets $V_1,\ldots,V_t$, such that for each pair $i,j$ with $i<j$, there is an edge of $G$
with one vertex in $V_i$ and the other in $V_j$. Define the {\it breadth} of the clique minor to be $\max_i |V_i|$.

\begin{lemma} \label{lemforrand}
Let $0<p<1$ be fixed, $0<\epsilon<1$, and define  $d:=\sqrt{(1-\epsilon)\log_b n}$ with $b=1/(1-p)$. Almost surely, the largest clique minor in $G(n,p)$ of breadth at most $d$ has order at most $4n^{1-\epsilon}d\ln n$.
\end{lemma}
\begin{proof}
If $d<1$, this trivially holds as there is no such nonempty clique minor of breadth at most $d$. Hence, we may assume $d \geq 1$. Consider a collection $C=\{V_1,\ldots,V_h\}$ of $h=\lceil 4n^{1-\epsilon}d\ln n \rceil$ nonempty vertex subsets each of size at most $d$. A rather crude estimate (which is sufficient for our purposes) on the number of such collections
is that it is at most $n^{dh}$. For each pair $V_i,V_j$, the probability there is an edge between $V_i$ and $V_j$ is
$$1-(1-p)^{|V_i||V_j|} \leq 1-(1-p)^{d^2} \leq e^{-(1-p)^{d^2}}=e^{-n^{\epsilon-1}},$$
where we used the inequality $1-x \leq e^{-x}$ for $0<x<1$ with $x=(1-p)^{d^2}$. 
 By independence, the probability
that there is, for all $1 \leq i < j \leq h$, an edge between $V_i$ and $V_j$ is at most $e^{-n^{\epsilon-1}{h \choose 2}}$.

Therefore, the expected number of clique minors of breadth at most $d$ and size at least $h$ is at most
$$n^{dh}e^{-n^{\epsilon-1}{h \choose 2}}=e^{h\left(d\ln n - n^{\epsilon-1}(h-1)/2\right)}=o(1).$$ 
This implies that almost surely no such clique minor exists. 
\end{proof}

Now we are ready to prove Theorem \ref{fractrandom}.

\vspace{0.1cm}
\noindent {\bf Proof of Theorem \ref{fractrandom}:} Let $G=G(n,p)$ be a random graph on $n$ vertices with edge density $p$. Let $b=1/(1-p)$ and
$\epsilon=4\frac{\log \log n}{\log n}$. Let $\mathcal{B}$ be a bramble for $G$. Suppose there is a weight function $w:\mathcal{B} \rightarrow \mathbb{R}_{\geq 0}$
such that $h = \sum_{B \in \mathcal{B}}w(B)$
and for each vertex $v$, the sum of the weights of the subgraphs in $\mathcal{B}$ containing $v$ is at most $1$.

Let $\mathcal{B}'$ denote the subcollection of subgraphs in $\mathcal{B}$ each with more $d=\sqrt{(1-\epsilon)\log_b n}$ vertices,
and $\mathcal{B}''=\mathcal{B} \setminus \mathcal{B}'$. We have
$$n \geq \sum_{B \in \mathcal{B}}w(B)|B| \geq \sum_{B \in \mathcal{B}'}w(B)|B| \geq d\sum_{B \in \mathcal{B}'}w(B),$$ where the first inequality follows from the fact that the sum of the weights of the subgraphs in $\mathcal{B}$ containing any given vertex is at most $1$. 
Hence, $\sum_{B \in \mathcal{B}'}w(B) \leq \frac{n}{d}$ and $$\sum_{B \in \mathcal{B}''}w(B) \geq h-\frac{n}{d}.$$

We now pick out a maximal subcollection of vertex-disjoint subgraphs in $\mathcal{B}''$.
We can greedily do this, picking out vertex disjoint subgraphs $B_1,\ldots,B_s$ until there are no more subgraphs in $\mathcal{B}''$
remaining which are vertex-disjoint from these subgraphs. Since the sum of the weight of all subgraphs containing a given
vertex is at most $1$, we must have $\sum_{i=1}^s |B_i| \geq h-\frac{n}{d}$. Since also $|B_i| \leq d$ for each $i$,
we have $s \geq \frac{h-n/d}{d}$. On the other hand, by Lemma \ref{lemforrand}, since $B_1,\ldots,B_s$ forms a clique minor of size $s$ and depth at most $d$,
almost surely $s \leq n^{1-\epsilon}d\ln n$. We therefore get almost surely
$$h \leq \frac{n}{d}+ds \leq \frac{n}{d}+n^{1-\epsilon}d^2 \ln n < (1+\epsilon)\frac{n}{\sqrt{\log_b n}},$$
where we use $n$ is sufficiently large, $n^{\epsilon}=\log^4 n$, $d=\sqrt{(1-\epsilon)\log_b n}$ and the estimate $\frac{1}{\sqrt{1-\epsilon}}<1+\frac{2}{3}\epsilon$
for $\epsilon<1/4$. As also $h_f(G) \geq h(G)$, and almost surely $h(G)=(1+o(1))\frac{n}{\sqrt{\log_b n}}$, this estimate completes the proof.
\qed

Note that there is an edge between each pair of connected subgraphs corresponding to the vertices of a clique minor. It follows that if $G$ is a graph with $m$ edges, then
$m \geq {h(G) \choose 2}$. We finish the section with a similar upper bound on the fractional Hadwiger number.

\begin{proposition}\label{propeasy}
If a graph $G$ has $m$ edges, then $h_f(G) \leq \sqrt{3m+1}$.
\end{proposition}
\begin{proof}
It is easy to see that we may assume that $G$ is connected and hence the number of vertices of $G$ is at most $m+1$. Let $\mathcal{B}$ be a bramble for $G$.
Suppose there is a weight function $w:\mathcal{B} \rightarrow \mathbb{R}_{\geq 0}$ such that $h = \sum_{B \in \mathcal{B}}w(B)$
and for each vertex $v$, the sum of the weights of the connected subgraphs in $\mathcal{B}$ containing $v$ is at most $1$.

Consider the sum $S=\sum w(B)w(B')$ over all ordered pairs of vertex-disjoint subgraphs in $\mathcal{B}$.
For any fixed subgraph $B$, the sum of the weights of the subgraphs in $\mathcal{B}$ containing at least one vertex in $B$ is at most $|B|$,
so the sum $\sum_{B'}w(B')$ over all subgraphs $B' \in \mathcal{B}$ disjoint from $B$ is at least $h-|B|$. Therefore,
$S \geq \sum_{B \in \mathcal{B}} w(B)(h-|B|)=h^2-\sum_{B \in \mathcal{B}} w(B)|B| \geq h^2-n$. For each edge $(i,j)$, the sum $\sum w(B)w(B')$ over all pairs of
vertex-disjoint subgraphs in $\mathcal{B}$ with $i \in V(B)$ and $j \in V(B')$ is at most $1$ since the sum of the weights of the subgraphs containing
any given vertex is at most $1$. As between each pair of vertex-disjoint subgraphs in $\mathcal{B}$ there is at least one edge,
we therefore get $S \leq 2m$. It follows $h \leq \sqrt{2m+n} \leq \sqrt{3m+1}$, which completes the proof.
\end{proof}

\section{Dense graphs with sublinear Hadwiger number}\label{sect3}

The purpose of this section is to give the details for the explicit construction of a dense graph with sublinear Hadwiger number.
We begin this section by proving Proposition \ref{simpleprop}, which states that  $$h(G[r])=r \cdot h_r(G) \leq r \cdot h_f(G)$$ holds for every graph $G$ and positive integer $r$. 

\vspace{0.1cm}
\noindent {\bf Proof of Proposition \ref{simpleprop}:} Let $G$ be a graph and $G[r]$ be the complete blow-up of $G$. Consider a maximum clique minor
in $G[r]$ of order $t=h(G[r])$ consisting of disjoint connected vertex subsets $V_1,\ldots,V_t$ with an edge between a vertex in $V_i$ and a vertex in $V_j$ for $i \not = j$. Let
$B_i$ be the vertex subset of $G$ where $v \in B_i$ if there is a vertex in the blow-up of $v$ which is also in $V_i$. The collection $\mathcal{B}=\{B_1,\ldots,B_t\}$ is clearly a
bramble. Define the weight $w(B_i)=1/r$ for each $i$. For each vertex $v$ of $G$, as $V_1,\ldots,V_t$ are vertex disjoint, at most $r$ sets $B_i$ contain $v$. Hence,
the bramble $\mathcal{B}$ with this weight function demonstrates $h_r(G) \geq h(G[r])/r$.

In the other direction, consider a bramble $\mathcal{B}$ for $G$ and a weight function $w$ on $\mathcal{B}$ such that $w(B)$ is a
multiple of $1/r$ for all $B \in \mathcal{B}$ and for every vertex $v$, the sum of the weights $w(B)$ over all $B$ containing $v$ is at most $1$. For
each such bramble $B$, we pick out $rw(B)$ copies of $B$ in the blow-up of $B$ in $G[r]$, such that all of the copies are vertex-disjoint. We can
do this since $rw(B)$ is a nonnegative integer, and for each vertex $v$ of $G$, the sum of $rw(B)$ over all $B \in \mathcal{B}$ which contain $v$ is at most $r$.
These copies of the sets in $\mathcal{B}$ form a clique minor in $G[r]$ of order $\sum_{B \in B} rw(B)=rh_r(G)$. Hence $h(G[r]) \geq rh_r(G)$, and we have
proved $h(G[r])/r=h_r(G)$. Since $h_r(G) \leq h_f(G)$, the proof is complete. \qed

The following theorem shows how to find, for each $0<\epsilon,p<1$, a graph $G$ of edge density at least $p$ such that the ratio of the Hadwiger number of $G$ to the number of vertices of $G$ is at most $\epsilon$ for $G$ and its blow-ups.

\begin{theorem}\label{thm0}
For each $0 < \epsilon,p<1$, there is a graph $G$ with edge density at least $p$ and $h_f(G) \leq \epsilon$. 
In particular, every complete blow-up of $G$ has edge-density at least $p$ and is $\epsilon$-Hadwiger. 
Moreover, for $p$ fixed and $\epsilon$ tending to $0$, the graph $G$ has $n_0=b^{\epsilon^{-2}+o(1)}$ vertices with $b=1/(1-p)$ and  can be found in time ${N \choose pN}^{1+o(1)}$ with $N={n_0 \choose 2}$.
\end{theorem}
\begin{proof}
From Theorem \ref{fractrandom}, we have that the random graph $G(n_0,p)$ almost surely has fractional
Hadwiger number $(1+o(1))\frac{n_0}{\sqrt{\log_b n_0}}$. Also, with at least constant positive probability, the edge density of such a random graph is at least $p$. Furthermore, Lemma \ref{lemforrand} shows that
$G(n_0,p)$ almost surely has the stronger property that its largest clique minor of depth at most $d=(1+\delta)\frac{n_0}{\sqrt{\log_b n_0}}$ with 
$\delta=4\frac{\log \log n_0}{\log n_0}$ has order less than
$s=4n_0^{1-\delta}d\ln n_0$. This is indeed stronger as in the proof of Theorem \ref{fractrandom}, we can bound the fractional Hadwiger number
from above by $n_0/d+ds$, which is less than $\epsilon n_0$ if the $o(1)$ term in the definition of $n_0$ is picked correctly. To show that a graph does not have a clique minor of depth at most $d$ and order $s$, it suffices to simply test all possible disjoint 
vertex subsets $V_1,\ldots,V_s$ with $|V_i| \leq d$ for $1 \leq i \leq s$, and check if each $V_i$ is connected and there is an edge between each $V_i$ and $V_j$ for $i \not = j$.
There are at most $n_0^{ds}$ such $s$-tuples of subsets to try.

Thus, by testing each graph on $n_0$ vertices with edge density $p$ for a clique minor of order $s$ and depth at most $d$, we will find such a graph $G$ without a clique minor order $s$ and depth at most $d$, and this is the desired graph $G$. The number of labeled graphs on $n_0$ vertices with edge density $p$ is
${N \choose pN}$ with $N={n_0 \choose 2}$ The amount of time, roughly $n_0^{ds}$, needed to test each such graph is a lower order term.
\end{proof}

If we wish to get an explicit construction of a dense graph which is $\epsilon$-Hadwiger on a given number $n$ of vertices, if $n$ is not a multiple of $n_0$, 
we can take a slightly larger blow-up of a small graph on $n_0$ vertices, and simply delete a few vertices (less than $n_0$ vertices with at most one 
from each clique in the complete blow-up).

The next theorem gives a solution to Thomason's problem by explicitly constructing a dense graph, which is a blow-up of a small graph $G$, for which the Hadwiger number of the graph and its complement are both relatively small.

\begin{theorem}\label{thm1}
For all $0<\epsilon<1$ there is a graph $G$ on $n_0=2^{(1+o(1))\epsilon^{-2}}$ vertices which can be found in time $2^{(1+o(1))n_0^2/2}$ such that 
$\max(h_f(G),h_f(\bar G)) \leq \epsilon n_0$. In particular, every complete blow-up of $G$ and its complement are $\epsilon$-Hadwiger.
\end{theorem}
\begin{proof}
From Theorem \ref{fractrandom}, a graph on $n_0$ vertices picked uniformly at random almost surely has fractional
Hadwiger number $(1+o(1))\frac{n_0}{\sqrt{\log n_0}}$. Furthermore, Lemma \ref{lemforrand} shows that a graph on $n_0$ vertices picked uniformly at random almost surely
has the stronger property that its largest clique minor of depth at most $d=(1+\delta)\frac{n_0}{\sqrt{\log n_0}}$ with $\delta=4\frac{\log \log n_0}{\log n_0}$ 
has order less than $s=4n_0^{1-\delta}d\ln n_0$. As in the proof of Theorem \ref{fractrandom}, we can bound the fractional Hadwiger number
from above by $n_0/d+ds$, , which is less than $\epsilon n_0$ if the $o(1)$ term in the definition of $n_0$ is picked correctly. To show that a graph and its complement does not have a clique minor of depth at most $d$ and order $s$, it suffices to simply test all possible disjoint vertex subsets $V_1,\ldots,V_s$ with $|V_i| \leq d$ for $1 \leq i \leq s$,
and check if each $V_i$ is connected and there is an edge between each $V_i$ and $V_j$ for $i \not = j$. 
There are at most $n_0^{ds}$ such $s$-tuples of subsets to try.

Thus, testing each graph on $n_0$ vertices, we find the desired graph $G$ for which $G$ and 
its complement do not contain a clique minor order $s$ and depth at most $d$. The number of graphs on $n_0$ vertices is $2^{{n_0 \choose 2}}$, and the amount of time, roughly $n_0^{ds}$, needed to test each such graph is a lower order term.
\end{proof}

\section{Concluding remarks}

$\bullet$ We showed how to explicitly construct a dense graph on $n$ vertices with Hadwiger number $o(n)$. However, random graphs
show that such graphs exist with Hadwiger number $O(\frac{n}{\sqrt{\log n}})$. It remains an interesting open problem to construct such graphs.

\vspace{0.1cm}

\noindent $\bullet$ We conjecture that almost all graphs $G$ satisfy $h(G)=h_f(G)$, i.e., a random graph on $n$ vertices almost surely satisfies that
its Hadwiger number and fractional Hadwiger number are equal. We proved in Theorem \ref{fractrandom} that these numbers are asymptotically equal for almost all graphs. This conjecture is equivalent to showing that almost all graphs satisfy the the ratio of the Hadwiger number to the number of vertices
is equal for all blow-ups of the graph.

\vspace{0.1cm}

\noindent $\bullet$ Note that if $H$ is a minor of $G$, then $h_f(H) \leq h_f(G)$. It follows that the family $\mathcal{F}_C$ of graphs $G$ with $h_f(G) < C$ is closed under
taking minors. The Robertson-Seymour theorem implies that $\mathcal{F}_C$  is characterized by a finite list of forbidden minors. For each $C$, what is this
family? We understand this family for $C \leq 4$ as then $h(G) =h_f(G)$.

\vspace{0.1cm}

\noindent $\bullet$ As noted by Seymour \cite{Se}, it would be interesting to prove a fractional analogue of Hadwiger's conjecture, that $h_f(G) \geq \chi(G)$ for all graphs $G$.
As $h_f(G) \geq h(G)$, this conjecture would follow from Hadwiger's conjecture. This may be hard in the case of graphs of independence number $2$. 
For such graphs on $n$ vertices, $\chi(G) \geq n/2$, and so Hadwiger's conjecture would imply $h(G) \geq n/2$, but the best known lower bound \cite{Fo} on the 
Hadwiger number is of the form $h(G) \geq (\frac{1}{3}+o(1))n$. Improving this bound to $h(G) \geq (\frac{1}{3}+\epsilon)n$ 
for some absolute constant $\epsilon>0$ is believed to be a challenging problem, and it is equivalent to proving a similar lower bound for $h_f(G)$.

\vspace{0.1cm}

\noindent $\bullet$ Graph lifts are another interesting operation. An $r$-lift of a graph $G=(V,E)$ is the graph on $V \times [r]$, whose edge set is the union 
of perfect matchings between $\{u\} \times [r]$ and $\{v\} \times [r]$ for each edge $(u,v) \in E$. Drier and Linial \cite{DL} studied clique minors in 
lifts of the complete graph $K_n$. One of the interesting open questions remaining here is whether every lift of the complete graph $K_n$ has Hadwiger number 
$\Omega(n)$. 

\vspace{0.1cm}

\noindent $\bullet$ Treewidth is an important graph parameter introduced by Robertson and Seymour \cite{RS84} in their proof of Wagner's conjecture.
A {\it tree decomposition} of a graph $G=(V,E)$ is a pair $(X,T)$, where $X = \{X_1, ..., X_t\}$ is a family of subsets of $V$,
and $T$ is a tree whose nodes are the subsets $X_i$, satisfying the following three properties.
\begin{enumerate}
\item $V=X_1 \cup \ldots \cup X_t$.
\item For every edge $(v,w)$ in the graph, there is a subset $X_i$ that contains both $v$ and $w$.
\item If $X_i$ and $X_j$ both contain a vertex $v$, then all nodes $X_z$ of the tree in the (unique) path between $X_i$ and $X_j$ contain $v$ as well.
\end{enumerate}

Robertson and Seymour proved that treewidth is related to the largest grid minor. Indeed, they proved that for each $r$ there is $f(r)$ such that every graph with treedwidth at least $f(r)$ contains a $r \times r$ grid minor. The original upper bound on $f(r)$ was
enormous. It was later improved by Robertson, Seymour, and Thomas \cite{RST94}, who showed $c r^2 \log r \leq f(r) \leq 2^{c'r^5}$ where $c$ and $c'$ are
absolute constants. In the other direction, it is easy to show that any graph which contains an $r \times r$ grid minor has treewidth at least $r$.

Separators are another important concept in graph theory which have many algorithmic, extremal, and enumerative applications. A vertex subset $V_0$ of a graph $G$ is a {\it separator} for $G$ if there is a partition $V(G)=V_0 \cup V_1 \cup V_2$
such that $|V_1|,|V_2| \leq 2n/3$ and there are no edges with one vertex in $V_1$ and the other vertex in $V_2$. A fundamental result of Lipton and Tarjan states that
every planar graph on $n$ vertices has a separator of size $O(\sqrt{n})$. This result has been generalized in many directions, to graphs embedded on a surface
\cite{GHT}, graphs with a forbidden minor \cite{AST}, intersection graphs of balls in $\mathbb{R}^d$, and intersection graphs of geometric objects in the plane
\cite{FPa}, \cite{FPb}. The {\it separation number} of a graph $G$ is the minimum $s$ for which every subgraph of $G$ has a separator of size at most $s$.

The bramble number of a graph $G$ is the minimum $b$ such that for every bramble for $G$ there is a set of $b$ vertices for which every subgraph in the bramble contains at
least one of these $b$ vertices.

Two graph parameters are {\it comparable} if one of them can be bounded as a function of the other, and vice versa. Robertson and Seymour
showed that treewidth and largest grid minor are comparable. The following theorem which we state without proof extends this result.
It may be surprising because some of these parameters appear from their definitions to be unrelated.

\begin{theorem}\label{connect}
Fractional Hadwiger number, $r$-integral Hadwiger number for each $r \geq 2$, bramble number, separation number, treedwidth, and maximum grid minor size are all comparable.
\end{theorem}

The dependence between some of these graph parameters is not well understood and improving the bounds remains an interesting open problem.

\vspace{0.1cm}

\noindent {\bf Acknowledgements:} I am greatly indebted to Noga Alon, Nati Linial, and Paul Seymour for helpful conversations.

\end{document}